\documentclass[11pt]{amsart}

\usepackage[T1]{fontenc}
\usepackage{mathptmx}
\usepackage{amssymb,epsfig,verbatim}
\usepackage{calligra,mathrsfs}
\theoremstyle{plain}
\usepackage[T1]{fontenc}
\usepackage[utf8]{inputenc}
\usepackage[english]{babel}

\usepackage{amssymb}
\usepackage{amsthm}
\usepackage{graphics}
\usepackage{amsmath}
\usepackage{amstext}
\usepackage{multirow}
\usepackage{enumerate}

\usepackage{hyperref}
\usepackage{graphicx}
\usepackage{amsmath}
\usepackage{amssymb}
\usepackage{amsthm}
\usepackage{amstext}
\usepackage{epstopdf}
\usepackage{mathrsfs}
\usepackage{amscd}
\usepackage{tikz}
\usetikzlibrary{cd}
\AtBeginDocument{%
   \def\MR#1{}
}

\newtheorem{thm}{Theorem}[section]
\newtheorem{lemma}[thm]{Lemma}
\newtheorem{prop}[thm]{Proposition}

\newtheorem{cor}[thm]{Corollary}

\newtheorem{THM}{Theorem}

\theoremstyle{remark}
\newtheorem{remark}[thm]{Remark}
\newtheorem{example}[thm]{Example}

\newcommand{\mb}{\mathbb}

\newcommand{\mc}{\mathcal}

\newcommand{\C}{\mb C}

\newcommand{\F}{\mc F}
\newcommand{\G}{\mc G}

\newcommand{\Mer}{\mathscr M}
\newcommand{\kerBott}{\mathscr X}
\newcommand{\E}{\mathsf E}

\newcommand\restr[2]{{
  \left.\kern-\nulldelimiterspace 
  #1 
  \vphantom{\big|} 
  \right|_{#2} 
  }}

\newcommand*\bola[1]{\tikz[baseline=(char.base)]{
    \node[shape=circle,draw,inner sep=0.5pt] (char) {\small{#1}};}}

\usepackage{todonotes}

\DeclareMathOperator{\codim}{codim}
\DeclareMathOperator{\sing}{sing}

\DeclareMathOperator{\Hom}{Hom}

\DeclareMathOperator{\trdeg}{tr\, deg_{\mathbb C}}

\newcommand{\ie}{{\it{i.e. }}}

\DeclareMathOperator{\Sym}{Sym}

\numberwithin{equation}{section}

\addtocounter{section}{0}             
\numberwithin{equation}{section}       

\title[Unlikely intersections of foliations]{Unlikely intersections of codimension one foliations}

\author{Gabriel Santos Barbosa}
\author{Jorge Vitório Pereira}

\date{\today}

\begin{document}

\begin{abstract}
    We study families of singular holomorphic foliations on complex projective manifolds whose total intersection defines a foliation of unexpectedly low codimension. 
\end{abstract}

\maketitle
\setcounter{tocdepth}{1}
\tableofcontents

\section{Introduction}
In \cite{cerveau2002pinceaux}, D. Cerveau proved that a pencil of codimension one singular holomorphic foliations on $\mathbb P^3$ is either the pull-back under a rational map of a pencil of singular holomorphic foliations on $\mathbb P^2$, or that every foliation in the pencil is singularly transversely affine. The present work originated from our attempts to understand and extend this result to broader settings. Our first observation is that the natural analogue of Cerveau's theorem holds on arbitrary projective manifolds.

\begin{THM}\label{THM:pencil}
    Let $X$ be a projective manifold and $\omega_0, \omega_1$ be two rational $1$-forms such that $\omega_0 \wedge \omega_1 \neq 0$. If for every $(s:t) \in \mathbb P^1$, the $1$-form $s \omega_0 + t \omega_1$  is integrable, defining a foliation $\F_{(s:t)}$, then at least one of the following assertions hold:
    \begin{enumerate}
        \item There exists a projective surface $Y$, a dominant rational map $\pi : X \dashrightarrow Y$, and foliations $\mathcal G_{(s:t)}$ on $Y$ such that
        $\F_{(s:t)} = \pi^* \G_{(s:t)}$ for every $(s:t) \in \mathbb P^1$.
        \item For every $(s:t) \in \mathbb P^1$, the foliation $\F_{(s:t)}$ is singularly transversely affine.
    \end{enumerate}
\end{THM}

As implicit in the statement, the two possibilities are not exclusive. Moreover, we point out the existence of literature aiming to determine when a pencil of foliations on a projective surface has all its elements singularly transversely affine, see for instance \cite{MR2136009} and \cite{MR4245610}.

Our proof of Theorem \ref{THM:pencil} closely follows Cerveau's original argument, with  minor modifications that allowed us to bypass his use of L\"uroth's theorem.

Given a finite set $S$ of pairwise distinct codimension one singular holomorphic foliations on a projective manifold $X$, their intersection define a singular holomorphic foliation with codimension at most the minimum of the cardinality of $S$ and $\dim X$. It is natural to ask  when this naive upper bound is not achieved. Theorem \ref{THM:pencil} provides an answer when the set $S$ has cardinality $3$. Indeed, linear algebra over the field of rational functions of $X$ tell us that the existence of three pairwise distinct foliations $\F_1, \F_2, \F_3$ such that the intersection of any two of them is the same codimension two foliation implies the existence of a pencil of integrable $1$-forms satisfying the assumptions of Theorem \ref{THM:pencil} and containing $\F_1, \F_2, \F_3$. Our main result reveals that such unexpected behaviour imposes strong restrictions on the foliations in the set $S$ also when $S$ has larger cardinality.

\begin{THM}\label{THM:main}
    Let $S = \{\mathcal F_1, \ldots, \mathcal F_{q+1}\}$ be a set of $q+1$ distinct singular codimension one foliations on a projective manifold
    $X$ of dimension $n > q$. Assume that $q \ge 3$ and that the intersection  of any $q$ distinct foliations in the set $S$ is the same codimension $q$ foliation $\mathcal G$. Let $\delta = \trdeg \C(X/\G)$ be the transcendence degree of the field of rational first integrals of $\G$.
    \begin{enumerate}
        \item If $\delta =q$ then there exists a dominant rational map $\pi : X \dashrightarrow Y$ to a $q$-dimensional projective manifold and $(q+1)$-foliations $\mathcal H_1, \ldots, \mathcal H_{q+1}$ foliations on $Y$  such that  $\mathcal F_i = \pi^*\mathcal H_i$ for every $i \in \{ 1, \ldots, q+1\}$.
        \item\label{I:between 2 and q-1} If $2\le \delta \le q-1$ then at least $q+1-\delta$ foliations in $S$ are (singularly) transversely projective.
        \item If $\delta \in \{ 0,1\}$ then every foliation in $S$ is (singularly) transversely projective.
    \end{enumerate}
\end{THM}

\begin{remark}
    Item (\ref{I:between 2 and q-1}) can be stated more precisely as follows. If $\delta=\trdeg \C(X/\G)$ is between $2$ and $q-1$ then there exists an integer $e$ between $0$ and $\trdeg \C(X/\G)$,  a dominant rational map $\pi : X \dashrightarrow Y$ to a $\delta$-dimensional projective manifold and $e$ foliations $\mathcal H_1, \ldots, \mathcal H_{e}$ on $Y$ such that, up to reordering, $\mathcal F_i = \pi^*\mathcal H_i$ for every positive integer $i\le e$ and every other foliation in $S$ is (singularly) transversely projective.
\end{remark}

We restricted ourselves to the case $q\ge 3$ in the statement of Theorem \ref{THM:main} since, as explained above, the case $q=2$ is already treated by Theorem \ref{THM:pencil}, with a stronger conclusion.

The proof of Theorem \ref{THM:main} goes through the study of the codimension $q$ foliation defined by the set $S$. It turns out that the abundance of foliations containing $\G$ implies strong restrictions on the transverse dynamics of $\G$. The next result appears as a key intermediate step toward the proof of Theorem \ref{THM:main}.

\begin{THM}\label{THM:Lie}
    Let $\G$ be a codimension $q$, $q \ge 2$, on a projective manifold $X$. Assume that $\trdeg(\C(X/\G)) = 0$, and that $\G$ is contained in $q+1$ codimension one foliations such that any $q$ of them have intersection equal to $\G$. Then we have the following alternative:
    \begin{enumerate}
        \item The set of foliations of codimension $r<q$ containing $\G$ is naturally identified with the Grassmanian of $r$-planes in $\C^q$ and every foliation containing $\G$ is an intersection of transversely affine codimension one foliations.
        \item The foliation $\G$ is a singularly transversely Lie foliation modeled over a non-abelian $q$-dimensional complex Lie algebra $\mathfrak{g}$ and the set of foliations of codimension $r<q$ containing $\G$ is naturally identified with the variety of Lie subalgebras of $\mathfrak{g}$ of codimension $r$. Furthermore, every foliation containing $\G$ is singularly transversely homogeneous.
    \end{enumerate}
\end{THM}

We stress that the assumption of Theorem \ref{THM:Lie} are  strong. Indeed,  \cite[Theorem 1]{MR2862041} shows that a general one-dimensional foliation is not contained in any higher dimensional foliation. 

As a consequence of Theorem \ref{THM:Lie}, we obtain the following  result.

\begin{THM}\label{THM:infinitas} 
    Let $\G$ be a codimension $q$ foliation.If $\G$ is contained in at least $q+1$ codimension one foliations, then $\G$ is contained in infinitely many of them.
\end{THM}

\subsection*{Structure of the paper}
In Section \ref{Section:basic}, we review the fundamental concepts of foliation and distribution theory on projective manifolds, with a particular focus on their fields of first integrals.

The technical core of this work is developed in Sections~\ref{SS:Bott} and~\ref{Section:invariant}. In Section~\ref{SS:Bott}, we examine Bott’s partial connection on the normal sheaf of a foliation and show that the space of its flat sections naturally inherits a Lie algebra structure. This observation leads us to explore transversely Lie foliations, as well as transversely parallelizable and transversely homogeneous foliations.

Section~\ref{Section:invariant} introduces further notions related to the Bott connection, including invariant distributions and eigenspaces arising from the action of Bott’s connection on the space of rational sections of the conormal sheaf and its exterior powers.

The remainder of the paper is dedicated to proving the results stated in the Introduction. In Section~\ref{Section:codim2}, we analyze intersections that give rise to codimension two foliations and establish Theorem~\ref{THM:pencil}. Section~\ref{Section:trdeg0} addresses the case in which the foliation induced by the intersection admits only constant rational first integrals, culminating in the proof of Theorem~\ref{THM:Lie}. Our main result, Theorem~\ref{THM:main}, is proved in Section~\ref{Section:transv_paral}. Finally, we conclude by proving Theorem~\ref{THM:infinitas} in Section~\ref{S:final}.

\subsection*{Acknowledgments} The authors express their appreciation for the financial support provided by CAPES/COFECUB and CNPq Projeto Universal 408687/2023-1 "Geometria das Equações Diferenciais Algébricas". Pereira  acknowledges the support  from CNPq (Grant number 304690/2023-6), and FAPERJ (Grant number E26/200.550/2023). Barbosa acknowledges the financial support provided by CNPq.

\section{Basic concepts}\label{Section:basic}

\subsection{Foliations}
Let $X$ be a projective manifold. A singular holomorphic foliation $\F$ on $X$ is
a pair $(T_{\F}, N^*_{\F})$ of subsheaves of $(T_X, \Omega^1_X)$ satisfying the properties listed below.
\begin{enumerate}
    \item The subsheaf $T_{\F}$ of $T_X$ is the annihilator of $N^*_{\F}$, that is  $T_{\F}$ is the kernel
    of the natural morphism
    \begin{align*}
        T_X & \longrightarrow \Hom_{\mathcal O_X}(N^*_{\F}, \mathcal O_X) \\
        v &\longmapsto  \{ \omega \mapsto \omega(v)  \} \, .
    \end{align*}
    \item Likewise, the subsheaf $N^*_{\F}$ of $\Omega^1_X$ is the annihilator of $T_{\F}$.
    \item\label{I:involutive} The sheaf $T_{\F}$ is involutive (closed under Lie brackets).
    \item\label{I:integrable} If the generic rank of $N^*_{\F}$ is $q$ then
    \[
        d \omega \wedge \omega_1 \wedge \ldots \wedge \omega_q = 0
    \]
    for any open set $U$ and any $(q+1)$ sections $\omega, \omega_1, \ldots, \omega_q$ of $\restr{N^*_{\F}}{U}$.
\end{enumerate}

Conditions (\ref{I:involutive}) and (\ref{I:integrable}) are equivalent to each other. When we drop both of them, we obtain the concept of a distribution on $X$.

The sheaf $T_{\F}$ is the tangent sheaf of $\F$ and the sheaf $N^*_{\F}$ is the conormal sheaf of $\F$.  The dimension $\dim \F$ of $\F$ is the
generic rank of $T_{\F}$ and the codimension $\codim \F$ of $\F$ is the generic rank of $N^*_{\F}$. Of course, $\dim X = \dim \F + \codim \F$.

Since $T_{\F}$ is a kernel of a morphism of reflexive $\mathcal O_X$-modules, it follows that $T_{\F}$ is not only a reflexive sheaf but also saturated inside $T_X$, that is $T_X/T_{\F}$ is torsion free. Likewise, $N^*_{\F}$ is reflexive and saturated in $\Omega^1_X$.

The singular set of $\F$ is the set of points  where $T_X / T_{\F}$ is not locally free which coincides with the set points where $\Omega^1_X/N^*_{\F}$ is not locally free. Succinctly, $\sing(\F) = \sing(T_X/T_{\F}) = \sing(\Omega^1_X/N^*_{\F})$. 

If $\F$ is a codimension $q$ foliation then Frobenius theorem implies that for every $x \in X - \sing(\F)$ there exists a germ of holomorphic submersion $f: (X,x) \to \mathbb (\C^q,0)$ such that $\restr{T_{\F}}{(X,p)}$ coincides with the relative tangent sheaf of $f$. The analytic continuation of the level sets of $f$ induce a decomposition of $X -\sing(\F)$ into pairwise disjoint codimension $q$ immersed connected analytic subvarieties, commonly referred to as the leaves of $\F$. 

We will denote the dual of $T_{\F}$ by $\Omega^1_{\F}$ and the dual of $N^*_{\F}$ by $N_{\F}$. Beware that some authors use different notations. For instance, some use $N_{\F}$ to stand for $T_X/T_{\F}$. On $X-\sing(\F)$,  $T_X/T_{\F}$ equals $N_{\F}$ but they differ on neighborhoods of points in the singular set.

\subsection{Differential forms defining foliations}
Given a codimension $q$ foliation $\F$, the wedge product of differential forms induces   an inclusion
\begin{equation}\label{E:inclusion}
    0 \to \det N^*_{\F} \to \Omega^q_X
\end{equation}
which yields a twisted $q$-form $\omega \in H^0(X, \Omega^q_X \otimes \det N_{\F})$ with coefficients in the line bundle $\det N_{\F} = (\det N^*_{\F})^*$.

The tangent sheaf of $\F$ can be recovered as the kernel of the natural morphism
\[
    i_{\omega} : T_X \to \Omega^{q-1}_X \otimes \det N_{\F} \, ,
\]
defined by the contraction of vector fields with $\omega$.

In general, given $\omega \in H^0(X, \Omega^q_X \otimes N)$, the kernel of the natural morphism $i_\omega$ defined by contraction with $\omega$ may have rank smaller than $\dim X - q$. For instance, if $\omega$ is a holomorphic symplectic $2$-form then the kernel of $i_{\omega}$ is the zero sheaf. The kernel of $i_{\omega}$ has rank $\dim X - q$ if, and only if, the germ of $\omega$ at a general point of $x \in X$ can be written as the wedge product of $q$ germs of $1$-forms tensor a nowhere vanishing section of $N$, see \cite{MR1842027}.

If $\sigma$ is a non-zero rational section of the line bundle $\det N^*_{\F}$  then
$\det N^*_{\F} = \mathcal O_X((\sigma)_0 - (\sigma)_\infty)$ and from the inclusion (\ref{E:inclusion})
one obtains a rational $q$-form, still denoted by $\omega$, with polar divisor equal to $(\sigma)_{\infty}$ and zero divisor equal to $(\sigma)_0$. Note that we can recover $T_{\F}$ from the rational $q$-form exactly as in the case of twisted $q$-forms. Moreover, if $\omega$
and $\omega'$ are two rational $q$-forms with the same kernels, and both have  rank $\dim X - q$, then the two rational $q$-forms differ by multiplication by a non-zero rational function.

\subsection{The field of rational first integrals and the Zariski closure of a foliation}\label{SS:field}
Let $\F$ be a foliation on a projective manifold $X$. A rational first integral for $\F$ is a rational function of $X$ with differential annihilated by any local section of $T_{\F}$. In other words, a rational first integral for $\F$  is a  rational function constant along the leaves of $\F$. It follows from its definition, that the set of rational first integrals of $\F$ is a subfield  $\mathbb C(X/\F)$ of $\mathbb C(X)$, the field of rational functions of $X$.

\begin{prop}\label{P:corpo}
    Let $X$ be a projective manifold of dimension $n$ and $\mathcal F$ be a foliation on $X$ of codimension $q$.
    The field $\C(X/\F)$ is algebraically closed in $\C(X)$ and the field extension $\C(X/\F) / \C$ has transcendence degree bounded by the codimension of $\F$. Consequently, there exists a projective manifold $Y$ and a dominant rational map $\pi : X \dashrightarrow Y$ with irreducible general fiber such that $\C(X/\F) = \pi^* \C(Y)$.
\end{prop}
\begin{proof}
     This result is folkloric. Due to the lack of an appropriate reference, we sketch its proof. To verify that $\C(X/\F)$ is algebraically closed in $\C(X)$ we must check that for any non-zero polynomial $p \in \C(X/\F)[t]$ and any solution $f \in \C(X)$ such that $p(f)=0$, we must have that that $f \in \C(X/\F)$. If $p$ has degree one, i.e. $p(t) = a_1 t + a_0$ for some $a_0 , a_1 \in \C(X/\F)$ then $p(f)=0$ implies $f = -a_1/a_0 \in \C( X/\F)$ as wanted. If the degree of $p$ is $e >1$,  we differentiate $0= p(f) = \sum_{i=0}^e a_i f^i$, $a_i \in \C(X/\F)$), along the leaves of $\F$ to obtain
    \[
        \sum_{i=0}^e ( \restr{d a_i}{T_{\F}}) f^i + \sum_{i=1}^e ia_i f^{i-1} \restr{d f}{T_{\F}} = 0 \, .
    \]
    Since $\restr{da_i}{T_{\F}} = 0$, we obtain that either $\restr{d f}{T_{\F}} = 0$ (as wanted) or $\sum_{i=1}^e ia_i f^{i-1}=0$. Hence $f$ satisfies a polynomial equation with coefficients in $\C(X/\F)$ of degree $e-1$. We conclude by induction that $\C(X/\F)$ is algebraically closed in $\C(X)$.
\end{proof}

Let $\overline{\F}$ be the foliation defined by the intersections of level sets of all rational $\mathbb C(X/\F)$. More precisely, $\overline{\F}$ is the foliation with tangent sheaf $T_{\overline{\F}}$ equal  to
\[
    \{ v \in T_X \, | \, v(f) =0 \text{ for every } f \in \C(X/\F) \} \, .
\]
We will say that the foliation  $\overline{\F}$ is the Zariski closure of $\F$.

A result by P. Bonnet \cite{MR2223484}     says that the general leaf $L$ of $\overline{\F}$ is obtained by taking the Zariski closure of a general leaf of $\F$ contained in $L$.

\section{Bott's partial connection and infinitesimal symmetries}\label{SS:Bott}

\subsection{Partial connections}\label{SS:partial} We start this section by recalling the definition of a partial holomorphic connection. Let $\F$ be a holomorphic foliation on a manifold $X$, and let $\mathcal E$ be a coherent sheaf of $\mathcal O_X$-modules.
We will write $\Omega^1_{\F}(\mathcal E)$ for the sheaf of $\mathcal O_X$-modules $\Hom_{\mathcal O_X}( T_{\F},  \mathcal E )$. A $\F$-partial connection on $X$ is a morphism of abelian sheaves
\[
    \nabla : \mathcal E \to \Omega^1_{\F}(\mathcal E) 
\]
which is $\C$-linear and satisfies Leibniz' rule:
\[
    \nabla ( f \sigma) = f \nabla(\sigma) + \restr{df}{T_{\F}} \otimes \sigma \, ,
\]
where $f \in \mathcal O_X$ and $\sigma \in \mathcal E$. In the expression above, and through out the text, we will abuse the notation and for any $\eta \in \Omega^1_{\F}$ and any $\sigma \in \mathcal E$, we will write $\eta \otimes \sigma$ for the image of  $\eta \otimes \sigma \in \Omega^1_{\F} \otimes \mathcal E$ under the natural morphism 
\begin{align*}
   \Omega^1_{\F} \otimes \mathcal E & \longrightarrow \Omega^1_{\F}(\mathcal E) \\
   \eta \otimes \sigma & \mapsto \{ v \mapsto \eta(v) \sigma \} \, .
\end{align*}
Under suitable assumptions, for instance if $\mathcal E$ or $\Omega^1_{\F}$ is locally free, this natural morphism is an isomorphism, but there are examples where it is not, see \cite[Remark 3.1 and Example 3.4]{fazoli2025jetsflatpartialconnections}.
Nevertheless, this subtlety will play no role in the arguments used in this paper, as we  will be interested on the action of $\nabla$ on global rational sections of $\mathcal E$. That is, we will study the linear map
\[
    H^0(X, \mathcal E \otimes \Mer_X) \to H^0(X, \Omega^1_{\F}(\mathcal E) \otimes \Mer_X) = H^0(X, \Omega^1_{\F} \otimes \mathcal E \otimes \Mer_X) 
\]
induced by $\nabla$ on rational sections. Here, and throughout, $\Mer_X$ stands for the sheaf of rational functions on $X$. Notice that the sheaves $\Omega^1_{\F}(\mathcal E) \otimes \Mer_X$ and $\Omega^1_{\F} \otimes \mathcal E \otimes \Mer_X$ are naturally isomorphic because $\Omega^1_{\F}(\mathcal E)$ and $\Omega^1_{\F} \otimes \mathcal E$ coincide over the smooth locus of $\F$. 

As in the case of usual connections, a $\F$-partial connection on a coherent sheaf $\mathcal E$ naturally induces $\F$-partial connections on $\mathcal E^*, \Sym^k \mathcal E$, $\wedge^k \mathcal E$. For instance, if $\nabla$ is a $\F$-partial connection on $\mathcal E$, then one naturally obtains a $\F$-partial connection $\nabla^*$ on $\mathcal E^*$ by imposing the product rule: 
\[
    d \langle \xi , \sigma \rangle = \langle \nabla^*(\xi) , \sigma \rangle + \langle \xi, \nabla(\sigma) \rangle \, ,
\]
where $\sigma \in \mathcal E$ and $\xi \in \mathcal E^*$.

\subsection{Bott's partial connection}\label{SS:Bott partial normal}
Let $\xi$ be a local section of $T_X/T_\F$. If $\hat \xi$ is a vector field  which projects to $\xi$ under $T_X \to T_X / T_{\F}$ then,
for any local section $v$ of $T_\F$, the projection of  $[v,\hat \xi ]$ to $T_X/T_\F$ does not depend on the
choice of the lift $\hat \xi$ .  This allowed Bott to define the partial connection
\begin{align*}
    \nabla : T_X/T_\F & \longrightarrow \Omega^1_{\F} \left( T_X/T_\F \right) \\
    \xi & \mapsto \left( v \mapsto [v, \hat \xi] \mod T_\F \right)
\end{align*}
nowadays called Bott's partial connection.

Thanks to the reflexiveness of $N_\F$ we can extend this operation to local
sections of $N_\F$ and obtain a partial connection $\nabla : N_{\F} \to \Omega^1_{\F} (N_{\F} )$, which we will also
call Bott's partial connection.

The restriction of $\nabla$ to a leaf $L$ of $\F$ is a flat holomorphic connection on $L$ with monodromy given by the linearization of the holonomy of $\F$ along $L$.

Let $U \subset X$ be any Zariski open subset and let $\xi \in N_{\F}(U)$ be a local section of the normal sheaf of $\F$.
Consider any lift $\hat \xi$ of $\xi$ to $T_X(U) \otimes \Mer_X(U)$.
As we are allowing rational coefficients in the lift $\hat \xi$, its existence is guaranteed . The condition $\nabla(\xi)=0$ is equivalent to the following statement:
for any rational section $v$ of $T_\F$, the bracket $[\hat \xi, v]$ is a rational section of $T_\F$.
Hence, the local flow  $\hat{\xi}$ (wherever defined) sends  leaves  of $\F$
to leaves of $\F$, that is $\hat \xi$ is an infinitesimal symmetry of $\F$.

\begin{prop}\label{P:bracket}
    If $\F$ is a foliation of codimension $q$ on a projective manifold $X$
    then the set of rational sections of $N_\F$ which are flat for Bott's connection is a finite
    dimensional vector space $\kerBott(X/\F)$  over $\C (X/\F)$ whose dimension is bounded by $q$ and carries
    a natural structure of $\C$-Lie algebra compatible with brackets of rational
    vector fields on $T_X$.
\end{prop}
\begin{proof}
    Let $\xi_1, \xi_2$ be two rational sections of  $N_\F$ satisfying $\nabla(\xi_1)=\nabla(\xi_2)=0$ and
    consider lifts $\hat \xi_1, \hat \xi_2$  to rational  sections of $T_X$. If $v$ is a rational section
    of $T_X$ then
    \[
        [ [\hat \xi_1, \hat \xi_2] , v ] = [ \hat \xi_1 , [ \hat \xi_2, v] ] - [\hat \xi_2, [ \hat \xi_1 , v] ]
    \]
    thanks to Jacobi's identity. It follows that the bracket $[\hat \xi_1, \hat \xi_2]$ is an infinitesimal
    symmetry of $\F$. Moreover, the analogue bracket obtained through different choices of lifts will differ
    from this one by a rational section of $T_\F$. This endows the set $\kerBott(X/\F)$ of flat rational sections of Bott's partial connection
     with a structure of $\mathbb C$-Lie algebra.

    Notice that for any $f \in \mathbb C(X/\F)$ and any $\xi$ rational section of $N_\F$ with $\nabla(\xi) = 0$
    we have
    \[
        \nabla( f \xi ) =f \nabla(\xi)  + \restr{df}{T_{\F}}  \otimes \xi  = 0
    \]
    This shows that $\kerBott(X/\F)$  is a vector space over $\C(X/\F)$.  Moreover, assume that we have $k+1$ elements $\xi_1, \ldots, \xi_k,\xi_{k+1}$  of $\kerBott(X/\F)$ such that  the corresponding rational sections of $N_\F$ satisfy $\xi_1 \wedge \ldots \wedge \xi_k \neq 0$ and $\xi_{k+1} = \sum_{i=1}^k a_i \xi_i$ for suitable rational functions $a_1, \ldots, a_k \in \mathbb C(X)$. Therefore
    \[
        0 = \nabla(\xi_{k+1}) =  \sum_{i=1}^k \restr{da_i}{T_{\F}}  \otimes \xi_i \, .
    \]
    It follows that $a_1, \ldots, a_k \in \mathbb C(X/\F)$. Since $N_\F$ has rank $q$, it follows that $\kerBott(X/\F)$ is a $\mathbb C(X/\F)$--vector space of dimension at most $q$.
\end{proof}

\subsection{Transversely Lie foliations}
Let $\F$ be a  codimension $q$ foliation on a projective manifold $X$ and let $\mathfrak g$ be a finite dimensional Lie algebra over $\C$.
A rational transverse $\mathfrak g$-structure for $\F$ is a rational $1$-form $\Omega$ taking values in $\mathfrak g$ such that
\begin{enumerate}
	\item\label{I:lie1} the Lie algebra $\mathfrak g$ has dimension $q$;
	\item\label{I:lie2}  $T_\F$,  the tangent sheaf of $\F$,  coincides with the kernel of morphism
	\[
		T_X \longrightarrow \mathfrak g \otimes \Mer_X
	\]
	defined by contraction with $\Omega$; and
	\item\label{I:lie3} the $1$-form $\Omega$ satisfies the integrability condition
	$
		d \Omega + \frac{1}{2} [\Omega, \Omega ] =0 \, .
	$
\end{enumerate}

Let ${\Omega}'$ be another $\mathfrak g$-valued $1$-form fulfilling properties (\ref{I:lie1}), (\ref{I:lie2}), (\ref{I:lie3}) above. We will say that ${\Omega}'$ and $\Omega$ define the same transverse $\mathfrak g$-structure for $\F$ if there exists a Lie algebra automorphism $\varphi$ of $\mathfrak g$ such that ${\Omega}'=\varphi(\Omega)$.

A  foliation $\F$ endowed with a transverse $\mathfrak g$-structure $\Omega$ will be called a (singularly) transversely Lie foliation modeled over $\mathfrak g$.  Beware that the same foliation may admit several transverse $\mathfrak g$-structures.

Let $\chi$ be a $\mathbb C$-Lie subalgebra of $\kerBott(X/\F)$ isomorphic to $\mathfrak g$ via $\psi : \chi \to \mathfrak g$.
Assume that $\chi$ has non trivial determinant in $\bigwedge^q N_\F \otimes \mathbb C(X)$.
These assumptions implies the existence of a unique $1$-form satisfying (\ref{I:lie1}), (\ref{I:lie2}), (\ref{I:lie3}) and such that for any $v\in \chi$, $\Omega (v)=\psi(v)$. Note that conversely, the datum of $\Omega$ as above defines unambiguously a Lie subalgebra ${\mathfrak g }_\Omega$ of $\kerBott(X/\F)$ isomorphic to $\mathfrak g$, namely the subset of rational sections of $N_\F$ determined by $\Omega^{-1} (\mathfrak g\otimes 1) \mod T_{\F}$.

We  summarize this discussion in the next proposition.

\begin{prop}\label{P:transversely Lie}
    Let $\F$ be a foliation of codimension $q$ on a complex manifold $X$ and $\mathfrak g$ be a $q$-dimensional Lie algebra. The set of transverse $\mathfrak g$-structures for $\F$ corresponds to the  $\mathbb C$-Lie subalgebras of $\kerBott(X/\F)$ isomorphic to $\mathfrak g$  and with non zero determinant  in $\bigwedge^q N_\F \otimes \mathbb C(X)$.
\end{prop}

Every time that we say that $\F$ is a transversely Lie foliation modeled over $\mathfrak g$, we will tacitly fix a transverse $\mathfrak g$-structure for $\F$.

\begin{cor}
    If the field of rational first integrals for $\F$ is $\mathbb C$ then $\F$ admits at most one transverse $\mathfrak g$-structure.
\end{cor}

If $\xi_1, \ldots, \xi_q$ is a basis for the Lie algebra $\mathfrak g$ satisfying
$
    [\xi_i, \xi_j ] = \sum_{k=1}^q \lambda_{ij}^k \xi_k
$
then the $1$-form $\Omega$ can be explicitly written as
$
    \Omega = \sum_{i=1}^q \omega_i \otimes \xi_i
$
where $\omega_i$ are rational $1$-forms on $X$. The integrability condition is then equivalent to the identities
\[
    d \omega_k = - \frac{1}{2} \sum_{i,j=1}^q \lambda_{ij}^k \omega_i \wedge \omega_j \, ,
\]
where $k$ ranges from $1$ to $q$.

Let $D$ be the polar divisor of $\Omega$ and let $G$ be a connected Lie group with Lie algebra $\mathfrak g$.
Darboux's fundamental theorem of calculus \cite[Chapter 3, Theorem 7.14]{MR1453120} implies that at any simply connected
analytic neighborhood $U \subset X- D$ of a point  $ x \in X- D$ there exists a holomorphic map $F : U \to G$ such that $\Omega$ coincides with the pull-back of the Maurer-Cartan $1$-form on $G$. Moreover, any two maps $F, \tilde F : U \to G$  with this property differ by left multiplication by an element
of $G$.

\subsection{Transversely parallelizable foliations}
We will say that  of a foliation $\F$ is (singularly) transversely parallelizable if the  $\mathbb C(X/\F)$-vector space $\kerBott(X/\F)$
has dimension equal to the codimension of $\F$. 

\begin{remark}
    The terminology singularly transversely parallelizable  is adapted from  \cite{MR0653455}, see also \cite[Chapter 4]{MR0932463}. Throughout, we will drop the adjective singularly to abbreviate.
\end{remark}

As follows from the definitions, every transversely Lie foliations is also transversely parallelizable.  However, the converse does not hold in general.  Nonetheless, when $\mathbb C(X/\F) = \C$, the two concepts coincide.

\begin{lemma}\label{L:transitivetransverse}
    Let $\F$ be a codimension $q$ foliation on a projective manifold $X$.
    The foliation $\F$ is transversely parallelizable  if, and only
    if, there exists rational $1$-forms $\omega_1, \ldots, \omega_q \in H^0(X,N^*_{\F} \otimes \Mer_X)$
    which generate $N^*_{\F}$ at a general point of $X$ and satisfy
    \begin{equation}\label{E:dualtransverse}
        d \omega_k = - \frac{1}{2} \sum_{i,j=1}^q \lambda_{ij}^k \omega_i \wedge \omega_j
    \end{equation}
    for a suitable collections of rational functions $\lambda_{ij}^k$ satisfying $\lambda_{ij}^k = -\lambda_{ji}^k$.
    Furthermore, the functions $\lambda_{ij}$ are rational first integrals for $\F$.
\end{lemma}
\begin{proof}
    Assume the foliation $\F$ is transversely parallelizable and let $\xi_1, \ldots, \xi_q \in H^0(X,N_\F \otimes \Mer_X)$
    be a basis of the $\mathbb C(X/\F)$-vector space $\kerBott(X/\F)$. Let also $\hat \xi_1, \ldots, \hat \xi_q \in H^0(X,T_X \otimes \Mer_X)$ be rational lifts of $\xi_1, \ldots, \xi_q$. As previously explained, We can write
    \begin{equation}\label{E:dualdual}
        [ \xi_i,\xi_j ] = \sum_{k=1}^q \lambda_{ij}^k \xi_k \quad \text{ or, equivalently, } \quad [ \hat \xi_i, \hat \xi_j ] = \sum_{k=1}^q \lambda_{ij}^k \hat \xi_k \mod T_{\F}
    \end{equation}
    where $\{\lambda_{ij}^k\}$ is a collection of rational functions, anti-symmetric on the lower indices. Let $\omega_1, \ldots, \omega_q \in H^0(X,N^*_\F \otimes \Mer_X)$ be rational sections of $N^*_\F$ dual to $\hat \xi_1, \ldots, \hat \xi_q$, \ie $\omega_i(\hat \xi_j) = \delta_{ij}$ (Kroenecker delta). We claim that this collection of $1$-forms satisfies Equation (\ref{E:dualtransverse}). To verify this fact, first notice that, for any choice of $i,j,k$,   $d \omega_i$ and $\omega_j \wedge \omega_k$  vanish when evaluated on $T_\F \otimes T_X$. This fact is clear for $\omega_j \wedge \omega_k$ but needs further elaboration for $d \omega_i$. Indeed, the integrability of $N^*_{\F}$ implies that $d\omega_i$ vanishes when evaluated on $T_\F \otimes T_\F$. Moreover, it also vanishes when evaluated at $T_{\F} \otimes \mathbb C \hat \xi_j$ thanks to Cartan's formula. It follows that $d\omega_i$ vanishes on $T_\F \otimes T_X$ as claimed. 
        
    To verify the validity of Equation (\ref{E:dualtransverse}), it remains to check the equality of both sides when evaluated on the elements $\hat \xi_i \otimes \hat \xi_j$. The result of the righthand side is clearly $- \lambda_{ij}^k$. To compute the lefthand side, apply Cartan's formula again and obtain
    \[
        d \omega_k ( \hat \xi_i, \hat \xi_j) = \hat \xi_i ( \omega_k( \hat \xi_j) ) - \hat \xi_j (\omega_k( \hat \xi_1)) - \omega_k([ \hat \xi_i, \hat \xi_j]) = -\lambda_{ij}^k
    \]
    as claimed. 

    If instead we start assuming the validity of Equation (\ref{E:dualtransverse}), we define $\xi_1, \ldots, \xi_q \in H^0(X,N_\F\otimes \Mer_X)$
    as a dual basis of $\omega_1, \ldots, \omega_q$, and let $\hat \xi_1, \ldots, \hat \xi_q \in H^0(X,T_X \otimes \Mer_X)$ be rational lifts as before.  Since
    \[
        L_{\hat \xi_\ell} \omega_k = i_{ \hat \xi_\ell} d \omega_k + d i_{ \hat \xi_\ell} \omega_k  = -\frac{1}{2} i_{ \hat \xi_\ell} \left( \sum_{i,j=1}^q \lambda_{ij}^k \omega_i \wedge \omega_j \right)
    \]
    clearly belongs to $H^0(X,N^*_\F\otimes \Mer_X)$, we deduce that $\xi_1, \ldots, \xi_q$ are infinitesimal symmetries of $\F$ which generate
    the $\C(X/\F)$-vector space  $\kerBott(X/\F)$. Moreover, Cartan's formula implies that they must satisfy Equation (\ref{E:dualdual}).

    To conclude, we proceed to prove that the rational functions $\lambda_{ij}^k$ are first integrals for $\F$. Fix $i,j,k$ and take the wedge
    product of Equation (\ref{E:dualtransverse}) with $\omega_1\wedge \cdots \wedge \widehat{\omega_i} \wedge \cdots \wedge \widehat{\omega_j} \wedge \cdots \wedge \omega_q$. Differentiate both sides,  observe that the lefthand side equals zero, and obtain that
    \[
        d \lambda_{ij}^k \wedge \omega_1 \wedge \cdots \wedge \omega_q=0
    \]
    as wanted.
\end{proof}

\subsection{Transversely homogeneous foliations}
Another variant of the concept of singularly transversely Lie foliations, central to our discussion, is the notion of singularly transversely homogeneous foliations. 

Let $\F$ be a foliation on a projective manifold $X$. A singularly transversely homogeneous structure for $\F$ consists of the following data:
\begin{enumerate}
    \item a rational $1$-form $\Omega$ with values in a finite-dimensional Lie algebra $\mathfrak g$ satisfying the Maurer-Cartan integrability condition $d \Omega + \frac{1}{2} [\Omega, \Omega ] =0$; and
    \item a Lie subalgebra $\mathfrak h \subset \mathfrak g$, such that $\dim \mathfrak g - \dim \mathfrak h = \codim \F$ and the tangent sheaf $T_{\F}$ is the kernel of the composition 
\[
    T_{X} \longrightarrow \mathfrak g \otimes \Mer_X \longrightarrow  \frac{\mathfrak g}{\mathfrak h}\otimes \Mer_X \, .
\]
where  the first map is defined by contraction with $\Omega$ and the second is the natural projection.
\end{enumerate}

We say that a foliation $\F$ is singularly transversely homogeneous if it admits a singularly transversely homogeneous structure.  As usual, we will tacitly omit the word singularly. 

\begin{remark}
    Usually, one considers transversely homogeneous structures as equivalence classes of $1$-forms with values in $\mathfrak g$ up to certain equivalence relations. As we are more interested in the existence of transversely homogeneous structures, we will by pass this discussion. Nevertheless, we will need to get back to it in the case of transversely homogeneous structures for codimension one foliations in the course of the proof of our last result (Theorem \ref{THM:infinitas}) in Section \ref{S:final}. 
\end{remark}

Assume now that $\F$ is a codimension one foliation. Let us present some possibilities for the pair $(\mathfrak g, \mathfrak h)$.

\begin{example}[Transversely additive foliations] 
    The Lie algebra $\mathfrak g$ is the one-dimensional Lie algebra and $\mathfrak h=0$. A transversely homogeneous structure for this pair is simply a closed rational $1$-form that defines the foliation. 
\end{example}

\begin{example}[Transversely affine foliations]
    The Lie algebra $\mathfrak g$ is the two-dimensional nonabelian Lie algebra $\mathfrak{aff}(\C)$ and $\mathfrak h$ is an one-dimensional abelian subalgebra different from the ideal $[\mathfrak g,\mathfrak g]$. A transversely homogeneous structure for a foliation $\F$ now consists of a rational $1$-form $\omega_0$ defining $\F$ and a closed rational $1$-form $\omega_1$  such that $d\omega_0 = \omega_0 \wedge \omega_1$.
\end{example}

\begin{example}[Transversely projective foliations]
    The Lie algebra $\mathfrak g$ is $\mathfrak{sl}(2,\C)$ and the subalgebra $\mathfrak h$ is a Borel subalgebra. Concretely, if we think of $\mathfrak{sl}(2,\C)$ as the $2\times 2$ matrices with zero trace, then we can take $\mathfrak h$ to be the lower triangular matrices. A transversely homogenous structure for $\F$  consists of a rational $1$-form $\omega_0$ defining $\F$ and  rational $1$-forms $\omega_1$ and $\omega_2$  such that 
    \begin{align*}
        d\omega_0 &= \omega_0 \wedge \omega_1 \, , \\
        d\omega_1 &= \omega_0 \wedge \omega_2 \, , \\
        d\omega_2 &= \omega_1 \wedge \omega_2 \, .
    \end{align*}
\end{example}

It turns out that these are, essentially, the only possibilities of transversely homogeneous structures for codimension one foliations. This is a consequence of  Tits' version \cite{MR120308} of a classical result by Lie stated below. 

\begin{lemma}[Lie-Tits]\label{L:LieTits} 
    Let $\mathfrak g$ be a complex finite-dimensional Lie algebra. If  $\mathfrak h \subset \mathfrak g$ is a codimension one subalgebra then there exists an ideal $\mathfrak I$ of $\mathfrak g$, contained in $\mathfrak h$, such that $\mathfrak g / \mathfrak I$ is a subalgebra of $\mathfrak{sl}_2$.
\end{lemma}

Indeed, if $(\mathfrak g, \mathfrak h)$ is a pair of Lie algebras with $\mathfrak h \subset \mathfrak g$ and $\mathfrak h$ of codimension one in $\mathfrak g$ then Lemma \ref{L:LieTits} above implies the existence of an ideal $\mathfrak I$ of $\mathfrak g$ such that the pair $(\mathfrak g/\mathfrak I, \mathfrak h/\mathfrak I)$ fits into one of the three examples above.

\subsection{Bott's partial connection on the conormal sheaf} 
If $\F$ is a germ of smooth codimension $q$ foliation and $\omega_1, \ldots, \omega_q$ is a basis of $N^*_{\F}$ then the integrability condition is equivalent to the existence of germs of $1$-forms $\eta_{ij}$ such that 
\[
    d \omega_i = \sum_{j=1}^q \eta_{ij} \wedge \omega_j , .
\]
for every $i \in \{ 1, \ldots, q\}$. In this setting, Bott's partial connection on $N^*_{\F}$, that is the connection on $N^*_{\F}$ determined by the connection presented in Subsection \ref{SS:Bott partial normal} can be made explicit through the formula
\[
    \nabla(\omega_i) = \sum_{j=1}^q \restr{ \eta_{ij} }{T_\F} \otimes \omega_j \, .
\]

It is convenient to interpret Lemma \ref{L:transitivetransverse} in terms of Bott's partial connection on  $N^*_{\F}$. It is saying that rational sections $\omega_1, \ldots, \omega_q$ of $N^*_{\F}$ which generate $N^*_{\F}$ at a general point of $X$ and which are flat for  Bott's partial connection satisfy Equation (\ref{E:dualtransverse}). 

We close this section by pointing out that Bott's partial connection on the normal/conormal of a foliation induces naturally partial connections of the wedge powers of the normal/conormal and also on their saturations.

\section{Foliations containing a given foliation}\label{Section:invariant}
Throughout this section, $X$ will denote a fixed projective manifold and  $\G$ will denote a codimension $q\ge 2$ foliation on $X$.

\subsection{Distributions containing a given foliation}
Let $\mathcal D$ be a distribution (or foliation) on $X$. We will say that $\G$ is contained in $\mathcal D$, or that $\mathcal D$ contains $\G$, if the following equivalent properties are satisfied:
\begin{enumerate}
    \item the tangent sheaf $T_{\mathcal D}$ of $\mathcal D$ contains the tangent sheaf $T_{\G}$ of $\G$; and
    \item the conormal sheaf $N^*_{\mathcal D}$ of $\mathcal D$ is contained in the conormal sheaf $N^*_{\G}$ of $\G$.
\end{enumerate}

Notice that for any $0<r<q$, the foliation $\G$ contains many codimension $r$ distributions: any non-zero product of $r$ rational sections of $N^*_{\G}$ defines a rational $r$-form whose kernel is the tangent sheaf of a codimension $r$ distribution containing $\G$.

\subsection{Invariant distributions} We now introduce a refined version of the previous definition, where we not only require that the conormal bundle of a given distribution is contained in the conormal bundle of $\G$, but also that it is invariant under Bott's partial connection. Specifically, we will say that a distribution $\mathcal{D}$ on $X$ is invariant by $\G$ if:
\begin{enumerate}
    \item the distribution $\mathcal D$ contains $\G$; and
    \item\label{I:invariance in conormal terms} Bott's partial connection on $N^*_{\G}$ maps $N^*_{\mathcal D}$ to $\Omega^1_{\G} (N^*_{\mathcal D}) \subset \Omega^1_{\G} (N^*_{\mathcal G}) $ .
\end{enumerate}

At this stage, it is important to mention that for a given foliation $\mathcal{G}$, the only invariant distributions might be $\mathcal{G}$ itself and the trivial distribution $\mathcal{D}$ with $T_{\mathcal{D}} = T_X$. In fact, as shown in \cite[Theorem 1]{MR2862041}, this property holds for very general one-dimensional foliations with sufficiently ample cotangent bundles on arbitrary projective manifolds.

For later reference, we record below a direct consequence of the definition of invariance.

\begin{lemma}\label{L:eta}
    If a distribution $\mathcal D$ of codimension $r$ defined by $\alpha \in H^0(X, \Omega^r_X \otimes \Mer_X)$ is invariant by
    $\mathcal G$ then there exists $\eta \in H^0(X, \Omega^1_{\G} \otimes \Mer_X)$ such that
    \[
        \nabla(\alpha) = \eta \otimes \alpha  \, ,
    \]
    where $\nabla$ is Bott's connection on $\wedge^r N^*_{\F}$.
\end{lemma}

Observe that the $\eta\in H^0(X, \Omega^1_{\G} \otimes \Mer_X)$ given by Lemma \ref{L:eta} is not intrinsically attached to the invariant distribution $\mathcal D$. It depends on the choice of  $\alpha$. A different choice, necessarily of the form $\alpha'=f \alpha$ for some $f \in \C(X)$, will satisfy $\nabla(\alpha') = \eta'\otimes \alpha'$ with
\[
    \eta' =  \eta + \restr{\frac{df}{f}}{T_{\G}} \, ,
\]
thanks to the Leibniz's rule.

\begin{prop}\label{P:G-invariance}
    If $\F$ is a foliation containing $\G$ then $\F$ is invariant by $\G$.
\end{prop}
\begin{proof}
    Assume that the codimension of $\F$ is $r$ and let $\omega_1, \ldots, \omega_r$ be rational sections of $N^*_{\F}$ such that $\omega_1 \wedge \cdots \wedge \omega_r \neq 0$. If  $\omega = \sum_{i=1}^r a_i \omega_i$ is any rational $1$-form vanishing on $T_{\F}$ then the integrability of $N^*_{\F}$ implies that
    \[
        d \omega = \sum_{i=1}^r \eta_i \wedge \omega_i \, ,
    \]
    for suitable rational $1$-forms $\eta_i$. It follows that Bott's connection for $\G$ on $N^*_{\G}$ applied to $\omega$  is equal to
    \[
        \nabla(\omega) = \sum_{i=1}^r \restr{\eta_i}{T_{\G}} \otimes \omega_i \, .
    \]
    The $\G$-invariance of $\F$ follows.
\end{proof}

In general, distributions invariant under a given foliation are not necessarily involutive, meaning they do not always define foliations. The following result presents a simple yet important exception to this principle. 

\begin{lemma}\label{L:invariant vs integrable}
    If $\mathcal D$ is a distribution invariant by a foliation $\G$ and $\dim \mathcal D = \dim \G + 1$ then $\mathcal D$ is a foliation.
\end{lemma}
\begin{proof}
    Let $\omega_1, \ldots, \omega_q$ be a basis of $\C(X)$-vector space $H^0(X, N^*_{\G}  \otimes \Mer_{X})$. 
    Since $\mathcal D$ contains the foliation $\G$, by the definition of invariance, and is such that $\mathcal \dim \mathcal D = \dim \G +1$, we can assume without loss of generality that $\omega_1, \ldots, \omega_{q-1}$ form a basis of the
    $\C(X)$-vector space $H^0(X, N^*_{\mathcal D} \otimes \Mer_{X})$. The integrability of $\G$ implies that we can write
    \[
        d \omega_k = \sum_{i=1}^q  \eta_{ki} \wedge \omega_i \, ,
    \]
    for suitable rational $1$-forms $\eta_{ki}$. Hence, if $\nabla$ is Bott's partial connection on $N^*_{\G}$, we can write
    \[
        \nabla(\omega_k) = \sum_{i=1}^q  \restr{\eta_{ki}}{T_{\G}} \otimes \omega_i \, ,
    \]
    The $\G$-invariance of $\mathcal D$ translates into $\restr{\eta_{kq}}{T_{\G}} =0$ for every $k \in \{ 1, \ldots, q-1\}$.
    Hence, in this range, $\eta_{kq}$ belongs to $H^0(X, N^*_{\G} \otimes \Mer_{X})$. It follows that, for $k \in \{1,\ldots, q-1\}$ we can write
    \[
        d \omega_k = \sum_{i=1}^{q-1}  \tilde{\eta}_{ki} \wedge \omega_i \, ,
    \]
    for suitable rational $1$-forms $\eta_{ki}$. It follows that $\mathcal D$ is a foliation as wanted.
\end{proof}

\subsection{Eigenpaces for Bott's connection}
Let $\G$ be a codimension $q$ foliation on a projective manifold $X$. For each integer $k$ between $1$ and $q$, and
each $\eta \in H^0(X, \Omega^1_{\G} \otimes \Mer_X)$ we will write $\E_{\G}^k(\eta)$ for the $\C(X/\G)$-vector space
\[
     \left\{ \omega \in H^0(X, \wedge^k N^*_{\G} \otimes \Mer_X) \, ; \, \nabla(\omega) = \eta \otimes \omega \right\} \, ,
\]
where $\nabla$ denotes  Bott's connection on $\wedge^k N^*_{\G}$. We will say that $\omega \in \E_{\G}^k(\eta)$ is an eigenvector for
Bott's connection with eigenvalue $\eta$, and that $\E_{\G}^k(\eta)$ is the corresponding $\eta$-eigenspace.

Multiplication of elements of $\E_{\G}^k(\eta)$ by a rational function $f \in \C(X)$ distinct from $0$ defines an isomorphism
\[
    \E_{\G}^k(\eta) \simeq \E_{\G}^k\left( \eta + \restr{\frac{df}{f}}{T_{\G}} \right) \, .
\]
Motivated by this isomorphism, we will say that two elements $\eta$ and $\eta'$ of $H^0(X, \Omega^1_{\G} \otimes \Mer_X)$ are log-equivalent
when there exists $f \in \C(X)$ such that
\[
    \eta - \eta'  = \restr{\frac{df}{f}}{T_{\G}}  \, .
\]
We will also say that an element $\eta$ of $H^0(X, \Omega^1_{\G} \otimes \Mer_X)$  is log-exact when it is log-equivalent to zero.

\begin{lemma}\label{L:E independente}
    Let $\ell \ge 1$ be an integer, and $\eta_1, \ldots, \eta_\ell \in H^0(X ,\Omega^1_{\G} \otimes \Mer_X)$ be pairwise non-log-equivalent
    elements. Then, for any $r$ between $1$ and $q$,  the natural linear map
    \[
        \bigoplus_{i=1}^\ell \E_{\G}^r(\eta_i) \otimes_{\C(X/\G)} \C(X) \to  H^0(X, \wedge^r N^*_{\G} \otimes \Mer_X)
    \]
    is injective.
\end{lemma}
\begin{proof}
    Assume the natural linear map is not injective. Then there exists elements $\omega_1, \ldots, \omega_k, \omega_{k+1}$ in
    $\bigoplus_{i=1}^\ell \E_{\G}^r(\eta_i)$ such that $\omega_1, \ldots, \omega_k$ are $\C(X)$-linearly independent and
    \begin{equation}\label{E:relacao}
        \omega_{k+1} = \sum_{i=1}^k a_i \omega_i
    \end{equation}
    for suitable rational functions $a_i \in \C(X)$, not all of them in $\C(X/\G)$.
    Moreover, we can further assume the existence of a map $\varphi: \{ 1, \ldots, k\} \to \{ 1, \ldots, \ell\}$ that for each  $\omega_i$, there exists $\eta_{\varphi(i)}$ such that $\nabla(\omega_i) = \eta_{\varphi(i)} \otimes \omega_i$.
    Applying $\nabla$ to Equation (\ref{E:relacao}) and subtracting from the result $\eta_{\varphi(k+1)} \otimes (\sum_{i=1}^k a_i \omega_i)$, we get the identity
    \[
        0 = \sum_{i=1}^k \left( \restr{da_i}{T_{\G}} + a_i ( \eta_{\varphi(i)} - \eta_{\varphi(k+1)}) \right) \otimes \omega_i \, .
    \]
    The $\C(X)$-linear independence of $\omega_1, \ldots, \omega_k$ implies that
    \[
        \restr{da_i}{T_{\G}} + a_i \left( \eta_{\varphi(i)} - \eta_{\varphi(k+1)} \right) = 0
    \]
    for every $i \in \{ 1, \ldots, k\}$. Since some $a_j$ does not belong to $\C(X/\G)$, we get that for this same $j$ that $ \eta_{\varphi(j)}$ and $\eta_{\varphi(k+1)}$ are not equal, but log-equivalent to each other, contradicting to our assumptions. The lemma follows.
\end{proof}

\begin{cor}\label{C:q+1}
    Let $S = \{\mathcal F_1, \ldots, \mathcal F_{q+1}\}$ be a set of $q+1$ distinct singular codimension one foliations on a projective manifold
    $X$ of dimension $n > q$. Assume that the intersection of any $q$ foliations in the set $S$ is the same codimension $q$ foliation $\mathcal G$.
    Then there exists $\eta \in H^0(X, \Omega^1_{\G}\otimes \Mer_X)$ such that
    \[
     \E_{\G}^1(\eta) \otimes_{\C(X/\G)} \C(X) \simeq  H^0(X,  N^*_{\G} \otimes \Mer_X) \, .
    \]
\end{cor}
\begin{proof}
    We can represent each $\F_i$ by a 1-form $\omega_i\in H^0(X, \Omega^1_{\G}\otimes \Mer_X)$. By integrability and Proposition \ref{P:G-invariance}, there exists $\eta_i \in H^0(X, \Omega^1_{\G}\otimes \Mer_X)$ such that $\omega_i\in \E^1_\G(\eta_i)$. Since the codimension of $\G$ is less than $q+1$, the $1$-forms $\omega_1,\ldots,\omega_{q+1}$ are linearly dependent over $\C(X)$. Therefore, there exist rational functions $f_i$ such that
    \[
    f_{q+1}\omega_{i+1} = \sum_{i=1}^q f_i\omega_i.
    \]
    If any of these functions were identically zero, we would obtain a dependence relation between $q$ or less $1$-forms. But this contradicts the assumption on the codimension of the intersection $q$ foliations in $S$. Moreover, as $f_i\omega_i$ also defines the foliation $\F_i$, we can assume, without loss of generality, that the $f_i$ are identically $1$.

    Then, repeating the arguments of the proof of Lemma \ref{L:E independente} for the relation
    \[
    \omega_{q+1} = \omega_1+\ldots+\omega_q,
    \]
    we conclude that there exists an $\eta$ such that $\omega_i \in \E^1_\G(\eta)$ for each $i\leq q+1$.
\end{proof}

\subsection{Foliations defined by eigenspaces} In this subsection, we will consider foliations $\G$ for which
the $\C(X)$-vector space $ H^0(X,  N^*_{\G} \otimes \Mer_X)$ is generated by one eigenspace for Bott's connection.
More explicitly, given  $\eta \in H^0(X, \Omega^1_{\G} \otimes \Mer_X)$, we will say that $\E_{\G}^1(\eta)$ defines, or generates,  $\G$  if
\[
    \E_{\G}^1(\eta) \otimes_{\C(X/\G)} \C(X) \simeq  H^0(X,  N^*_{\G} \otimes \Mer_X) \, .
\]
In this case, we will say that $\G$ is defined/generated by $\E_{\G}^1(\eta)$.

For later use, we point out below a straightforward consequence of Lemma \ref{L:E independente}

\begin{lemma}\label{L:unique generation}
    If $\E_{\G}^1(\eta)$ defines $\G$ then 
    \[
        \E_{\G}^k(k \eta) \otimes_{\C(X/\G)} \C(X) \simeq H^0(X, \wedge^k N^*_{\G} \otimes \Mer_X) \, .
    \]
    Moreover, if $\E_{\G}^k(\eta') \neq 0$ then $k\eta'$ is log-equivalent to $k \eta$.
\end{lemma}
\begin{proof}
    Bott's connection on $N^*_{\G}$ induces a connection on $\wedge^k N^*_{\G}$. If $\omega_1, \ldots, \omega_k$ are
    sections of $N^*_{\G}$ satisfying $\nabla(\omega_i) = \eta \otimes \omega_i$, it follows that
    \[
        \nabla(\omega_1 \wedge \cdots \wedge \omega_k) = k \eta \otimes \omega_1 \wedge \cdots \wedge \omega_k \, .
    \]
    Since $E^1(\eta)$ generates the $\C(X)$-vector space $H^0(X, N^*_{\G} \otimes \Mer_X)$, we obtain a surjective linear
    map $\E_{\G}^k(k \eta) \otimes_{\C(X/\G)} \C(X) \to H^0(X, \wedge^k N^*_{\G} \otimes \Mer_X)$. The injectivity follows from
    Lemma \ref{L:E independente}, which also implies that  if $\E_{\G}^k(\eta') \neq 0$ then $k\eta'$ is log-equivalent to $k \eta$.
\end{proof}

\begin{lemma}\label{L:escritura}
    Let $\G$ be a codimension $q$ foliation on a projective manifold $X$ and let $\eta$ be an element of $H^0(X, \Omega^1_{\G} \otimes \Mer_X)$.
    If $\E_{\G}^1(\eta)$ defines $\G$,  $\widehat{\eta}$ is a rational $1$-form on $X$ such that
    \[
        \restr{\widehat{\eta}}{T_{\G}} = \eta\, ,
    \]
    and $\omega_1, \ldots, \omega_q$ is a basis of $\E_{\G}^1(\eta)$, then there exists a collection  $\{ c_{ij}^k\}  \subset \C(X)$ of rational functions with $i,j,k \in \{ 1, \ldots, q\}$, anti-symmetric on the lower indices, such that
    \[
        d \omega_k  = \sum_{i,j=1}^q \lambda_{ij}^k \cdot \omega_i \wedge \omega_j + \widehat{\eta} \wedge \omega_k  \, ,
    \]
    for every $k \in \{ 1, \ldots, q \}$.
\end{lemma}
\begin{proof}
    Let $\omega_1, \ldots, \omega_q$ be a basis of the $\C(X/\G)$-vector space $\E_{\G}^1(\eta)$. Since $\E_{\G}^1(\eta)$ defines $\G$, we have that
    \[
        d \omega_k =  \sum_{j=1}^q \alpha_k^j \wedge \omega_j \, ,
    \]
    for suitable rational $1$-forms $\alpha_i^j$. Since $\nabla(\omega_k) = \eta \otimes \omega_k$, it follows that for $k \neq j$, $\restr{\alpha_k^j}{T_{\G}}=0$ and, consequently,  $\alpha_k^j$ is a rational section of $N^*_{\G}$. We can thus  write 
    \[
        \alpha_k^j = \sum_{i=1}^q a_{ij}^k \omega_i
    \]
    where $a_{ij}^k$ are suitable rational functions.  Likewise, $\alpha_k^k - \widehat{\eta}$ is also a rational section of $N^*_{\G}$  for every $k \in \{ 1, \ldots, q\}$, and we can write
    \[
        \alpha_k^k - \widehat{\eta} = \sum_{i=1}^q a_{ik}^k \omega_i \, .
    \]
    The result follows after setting $\lambda_{ij}^k  = a_{ij}^k - a_{ji}^k$.
\end{proof}

\begin{cor}\label{C:quasitransvLie}
    Let $\G$ be a codimension $q$ foliation on a projective manifold $X$. If $\G$ is defined by $\E_{\G}^1(0)$ then $\G$ is transversely parallelizable and the rational functions $\lambda_{ij}^k$ given by Lemma \ref{L:escritura} are rational first integrals for $\G$, \ie $\lambda_{ij}^k \in \C(X/\G)$.
\end{cor}
\begin{proof}
    It suffices to combine Lemma \ref{L:escritura} with Lemma \ref{L:transitivetransverse}.
\end{proof}

\subsection{Eigenspace associated to a non log-exact eigenvalue}

\begin{prop}\label{P:integravel}
    Let $\G$ be a codimension $q$ foliation on a projective manifold $X$. Assume the existence of $\eta \in H^0(X, \Omega^1_{\G} \otimes \Mer_X)$ such that $\G$ is defined by $\E_{\G}^1(\eta)$ . If $\eta$ is not log-exact then every $\omega \in \E_{\G}^1(\eta)$ is integrable.
\end{prop}
\begin{proof}
    Aiming at a contradiction, assume there exists $\omega \in \E_{\G}^1(\eta)$ which is not integrable.
    Without  loss of generality, assume that  $\omega$ is one of the elements of a basis $\omega_1, \ldots, \omega_q$ for $\E_{\G}^1(\eta)$, say $ \omega = \omega_1$. Expanding the Frobenius integrability condition for  $\omega_1$ using Lemma \ref{L:escritura}, we obtain:
    \[
        d \omega_1 \wedge \omega_1  =  \sum_{i,j=2}^q \left( c_{ij}^1 \cdot  \omega_i \wedge \omega_j \right) \wedge \omega_1  \, .
    \]
    The non-integrability of $\omega_1$  implies the existence of a pair $(a,b) \in \{ 2, \ldots, q\}$ with $a < b$,  such that  $c_{ab}^1\neq 0$, where $c_{ij}^k$ are the rational functions given by the expression for $d\omega$ given by Lemma \ref{L:escritura}.

    Set $\Omega$   to be the rational $q$-form $\omega_1 \wedge \cdots \wedge \omega_q$ and let $\Omega_{a,b}$ be the $(q-2)$-form obtained by the wedge product of $\omega_i$ for $i\in \{ 1, \ldots ,q\} - \{ a ,b\}$ in such an order that
    \[
        \omega_a \wedge \omega_b \wedge \Omega_{a,b} = \Omega \, .
    \]
    Lemma \ref{L:escritura} implies the identity  
    \[ 
        0 = d^2 \omega_1 = d \left( \sum_{i,j=1}^q c_{ij}^1 \cdot \omega_i \wedge \omega_j + \widehat{\eta} \wedge \omega_1 \right) \, ,
    \]
    which turns out to be equivalent to
    \[
         0  =   \sum_{i,j=1}^q \left( \underbrace{d c_{ij}^1 \wedge \omega_i \wedge \omega_j}_{\bola{A}} + \underbrace{c_{ij}^1 \cdot  d \omega_i \wedge \omega_j}_{\bola{B}} -  \underbrace{c_{ij}^1 \cdot  \omega_i \wedge d\omega_j}_{\bola{C}}  \right) + \underbrace{d \widehat{\eta} \wedge \omega_1}_{\bola{D}} - \underbrace{\widehat{\eta} \wedge d \omega_1}_{\bola{E}}\, .
    \]
    Wedge this expression with $\Omega_{a,b}$ to obtain the vanishing of
    \[
           \left( \underbrace{( dc_{ab}^1 - dc_{ba}^1) }_{\bola{A}} + \underbrace{( c^1_{ab} - c^1_{ba}) \cdot  \widehat{\eta}}_{\bola{B}}
        - \underbrace{( c^1_{ba} - c^1_{ab} )\cdot   \widehat{\eta} }_{\bola{C}} +
          \underbrace{0}_{\bola{D}} - \underbrace{( c^1_{ab} - c^1_{ba}) \cdot  \widehat{\eta}}_{\bola{E}} \right) \wedge \Omega \, .
    \]
    Therefore, the identity 
    \[
        \restr{\left(dc_{ab}^1 + c_{a,b}^1 \widehat{\eta}\right)}{T_{\G}} = 0 \implies \restr{\left(\frac{dc_{ab}^1}{c_{a,b}^1} + \widehat{\eta}\right)}{T_{\G}} = 0
    \]
    holds true, showing that $\eta$ is log-exact. This contradicts our assumptions.
    The proposition follows.
\end{proof}

\begin{cor}\label{C:zariski dense}
    Let $\G$ be a codimension $q$ foliation satisfying the assumptions of Proposition \ref{P:integravel}. If $\eta$ is not log-exact and $q \ge 3$ then $\C(X/\G) = \C$.
\end{cor}
\begin{proof}
    Let $\omega_1, \omega_2, \omega_3  \in \E_{\G}^1(\eta)$ such that $\omega_1\wedge \omega_2 \wedge \omega_3 \neq 0$. Let $f \in \C(X/\G)$.  Proposition \ref{P:integravel} implies that $\omega_1, \omega_2, \omega_3 , \omega_1 + f \omega_2$, $\omega_1 + f \omega_3$, and $\omega_2 + f \omega_3$ are all integrable $1$-forms. The integrability of $\omega_1, \omega_2$, and $\omega_1 + f \omega_2$ implies that
    \[
        df \wedge \omega_1 \wedge \omega_2 = 0 \, .
    \]
    Hence $f$ is a first integral for the foliation defined by $\omega_1 \wedge \omega_2$. Likewise, $f$ is a first integral for the foliations defined by $\omega_1 \wedge \omega_3$ and $\omega_2 \wedge \omega_3$. Since the tangent sheaves of these three codimension two foliations generate $T_X$ over the open subset where $\omega_1 \wedge \omega_2 \wedge \omega_3$ is holomorphic and different from zero, the function $f$ must be constant.
\end{proof}

\subsection{Algebraically integrable foliations}

\begin{prop}\label{P:algebraically integrable}
    Let $\G$ be a codimension $q$ foliation on a projective manifold. If $\trdeg \C(X/\G) =q$ then
    the $\C(X/\G)$-vector space $\E_{\G}^1(0)$ generates $\G$. Moreover, there exists a projective manifold $Y$ of dimension $q$,
    a dominant rational map $\pi :X \dashrightarrow Y$ with irreducible general fiber such that for any distribution $\mathcal D$
    invariant by $\G$, there exists a distribution $\mathcal D_Y$ on $Y$ such that $\mathcal D = \pi^* \mathcal D_Y$.
\end{prop}
\begin{proof}
    Let $f_1, \ldots, f_q \in \C(X/\G)$ be algebraically independent elements. Clearly their differentials belong to $\E_{\G}^1(0)$. Moreover, the algebraic independence implies that they generate a $q$ dimensional $\C(X)$-vector subspace of $H^0(X, N^*_{\G} \otimes \Mer_X)$. Therefore $\E_{\G}^1(0)$ generates $\G$ as claimed.

    Proposition \ref{P:corpo} implies the  existence of a projective manifold $Y$ and of a  rational map $\pi : X \dashrightarrow Y$, with irreducible general fiber, such that $\pi^* \C(Y) = \C(X/\G)$. Let $\mathcal D$ be a codimension $r$ distribution invariant by $\G$. Lemma \ref{L:unique generation} implies that $\mathcal D$ is defined by $\omega \in \E^r_{\G}(0)$. Consequently,  $\omega$ can be written as
    \[
        \sum_{I} a_I df_{i_1} \wedge \cdots \wedge df_{i_r}
    \]
    where $a_I \in \C(X/\G) = \pi^* \C(Y)$. Therefore $\omega$ is the pull-back under $\pi$ of a rational $r$-form on $Y$. The proposition follows.
\end{proof}

\section{Codimension two foliations}\label{Section:codim2}

This section is devoted to the proof of Theorem \ref{THM:pencil} of the introduction. We are actually going to establish the more precise version below.

\begin{thm}\label{T:Cerveau}
    Let $\G$ be a codimension two foliation on a projective manifold $X$. If there exists $\eta \in H^0(X,\Omega^1_{\G} \otimes \Mer_X)$ such that $\E_{\G}^1(\eta)$ generates $\G$ then one of the following statement holds true.
    \begin{enumerate}
        \item\label{I:Cerveau trdeg 2} there exists a projective surface $Y$ and a rational map $\pi: X \dashrightarrow Y$ such that $\G$ is defined by $\pi$ and every codimension one  foliation $\F$ containing $\G$ is of the form $\pi^* \F_Y$ for some foliation $\F_Y$ on $Y$; or
        \item\label{I:Cerveau trdeg 1} the field extension $\C(X/\G)/\C$ has transcendence degree equal to one, every codimension one foliation containing $\G$ is transversely affine, and the set of foliations containing $\G$ is parameterized by $\mathbb P^1(\C(X/\G))$; or
        \item\label{I:Cerveau trdeg 0} the general leaf of $\G$ is Zariski dense, every codimension one  foliation containing $\G$ is transversely affine, and the set of foliations containing $\G$ is parametrized by $\mathbb P^1(\C)$.
    \end{enumerate}
    Furthermore, if $\eta$ is not log-exact then Statement (\ref{I:Cerveau trdeg 0}) holds true.
\end{thm}

Before dwelling with the proof of Theorem \ref{T:Cerveau}, we show how to deduce Theorem \ref{THM:pencil} from it.

\begin{prop}
    Theorem \ref{T:Cerveau} implies Theorem \ref{THM:pencil}.
\end{prop}
\begin{proof}
    Let $\omega_0$, $\omega_1$ be rational $1$-forms as in the statement of Theorem \ref{THM:pencil}. To wit, for every $(s:t) \in \mathbb P^1$ the $1$-form $s\omega_0 + t \omega_1$ is integrable. Let $\G$ be the codimension two foliation defined by $\omega_0 \wedge \omega_1$. Corollary \ref{C:q+1}     implies the existence of $\eta \in H^0(X,\Omega^1_{\G} \otimes \Mer_X)$  such that $\omega_0 , \omega_1 \in \E^1_{\G}(\eta)$. We are therefore in position to apply Theorem \ref{T:Cerveau}. Theorem \ref{THM:pencil} follows.
\end{proof}

We will now proceed to proof Theorem \ref{T:Cerveau}. As already mentioned in the introduction, most of the arguments are due to Cerveau and appeared in \cite{cerveau2002pinceaux}, see also \cite{MR1285386}. The novelty in our argumentation appears in the second paragraph of Subsection \ref{SS:Proof Pencil} which deals with the case where $\trdeg \C(X/\G) =1$.

\subsection{Curvature of $3$-webs}
Under the assumptions of Theorem \ref{T:Cerveau}, Lemma \ref{L:invariant vs integrable} implies that any rational $1$-form in $\E_{\G}^1(\eta)$ is integrable. Given any three $1$-forms $\omega_1, \omega_2, \omega_3 \in \E_{\G}^1(\eta)$ defining three distinct foliations $\F_1, \F_2, \F_3$ and satisfying
\[
    \omega_1 + \omega_2 + \omega_3 = 0
\]
there exists a unique rational $1$-form $\theta$ such that
\[
    d \omega_i = \theta \wedge \omega_i \, .
\]

The three foliations define a $3$-web $\mathcal W = \F_1 \boxtimes \F_2 \boxtimes \F_3$, see \cite{MR3309231}. The differential of $\theta$, $\Theta = d \theta$, is intrinsically attached to $\mathcal W$ and is the so-called curvature of $\mathcal W$. Notice that
\[
    0 = d (d \omega_i) = \Theta \wedge \omega_i - \theta \wedge d \omega_i \implies \Theta\wedge \omega_i =0 \, .
\]
Therefore, if the $2$-form $\Theta$ is non-zero, it defines the codimension two foliation $\G = \cap_{i=1}^3 \F_i$ we started with.
If instead $\Theta$ vanishes identically then the every foliation in the $3$-web is tranversely affine. For more about web geometry, see \cite{MR3309231}. For the concept of transversely affine foliations, the reader can consult \cite{MR3294560} and references therein.

The key observation, due to Cerveau (see \cite[Proposition 2]{cerveau2002pinceaux} and \cite[Proposition 7]{MR1285386}), is the following result. We reproduce his arguments for reader's ease.

\begin{prop}\label{P:Cerveau}
    Notation as above. If the curvature of $\mathcal W$ is non-zero  then $\C(X/\G) \neq \C$.
\end{prop}
\begin{proof}
    Since $\Theta \wedge \omega_i =0$, there exists a non-zero rational function $a \in \C(X)$ such that
    \[
        d \theta = \Theta = a \omega_1 \wedge \omega_2.
    \]
    Differentiation leads to the identities
    \begin{align*}
        0 &= da\wedge \omega_1\wedge \omega_2 + a(d\omega_1\wedge\omega_2-\omega_1\wedge d\omega_2)\\
        &= (da + 2a\theta)\wedge\omega_1\wedge\omega_2.
    \end{align*}
    It follows the existence of rational $K_1,K_2 \in \C(X)$  such that
    \begin{equation}\label{E:a_theta}
        \frac{1}{2}\frac{da}{a}  +\theta= K_1\omega_1+K_2\omega_2.
    \end{equation}
    Note that if $K_1,K_2$ and $a$ were all constants, then $\theta$ would be an element of the pencil. Therefore $\Theta = d\theta = \theta\wedge\theta = 0$, contradicting our hypothesis.

    Taking the derivative of Equation (\ref{E:a_theta})  yields
    \[
        d\theta = (dK_1+K_1\theta)\wedge\omega_1 + (dK_2+K_2\theta)\wedge\omega_2.
    \]
    Wedging this expression  with $\omega_i$ gives
    \begin{align*}
        0 &= (dK_i+K_i\theta)\wedge \omega_1 \wedge \omega_2.\\
        0 &= \left(\frac{dK_i}{K_i} - \frac{1}{2} \frac{da}{a} \right) \wedge\omega_1\wedge\omega_2.\\
        0 &= d\left(\frac{K_i^2}{a}\right)\wedge\omega_1\wedge\omega_2.
    \end{align*}
    We conclude that $\frac{K_1^2}{a},\frac{K_2^2}{a}$ and $\frac{K_1}{K_2}$ are all first integrals for $\G$. It remains to show that one of these rational functions is non-constant. If not,
    \[
        \frac{1}{2} \frac{da}{a} + \theta = K_1 \omega_1+ K_2 \omega_2 = K_2(C \omega_1+ \omega_2),
    \]
    for some $C\in\C$. Then
    \[
        d\theta = dK_2\wedge(C\omega_1+\omega_2) + K_2\theta\wedge(C\omega_1+\omega_2)= \left(\frac{dK_2}{K_2}+\theta\right)\wedge K_2(C\omega_1+\omega_2).
    \]
    On the other hand, $d\left(\frac{K_2^2}{a}\right)=0$ implies that $\frac{1}{2}\frac{da}{a}= \frac{dK_2}{K_2}$. And thus $\Theta = d\theta=0$, contrary to our assumptions. The proposition follows.
\end{proof}

\subsection{Proof of Theorem \ref{T:Cerveau}}\label{SS:Proof Pencil} If $\trdeg(\C(X/\G)) = 2$ then Proposition \ref{P:algebraically integrable} implies Item (\ref{I:Cerveau trdeg 2}).

If instead $\trdeg(\C(X/\G)) =1$, let $f \in \C(X/\G)$ be a non-constant rational first integral for $\G$. Note that  $\omega_1 =  df$, the differential of $f$, is a flat rational section of $N^*_{\G}$. Therefore, $\omega_1 = df \in \E_{\G}^1(0)$. Lemma \ref{L:unique generation} implies that $\eta$ is log-exact and there is no loss of generality in assuming that $\E_{\G}^1(0)$ generates $\G$. If $\omega_2 \in \E_{\G}^1(0)$ is any element defining an invariant distribution different from the one defined by $\omega_1=df$, then Lemma \ref{L:escritura} implies that $\omega_2\wedge d\omega_2=0$ and therefore defines a foliation. Hence the set of foliations containing $\G$ can be identified with $\mathbb P(\E_{\G}^1(\eta))(\C(X/\G)) = \mathbb P^1(\C(X/\G))$. Consider the $3$-web $\mathcal W$ defined by $\omega_1, \omega_2$, and $\omega_3 = - \omega_1 - \omega_2$. Let $\theta$ be the unique rational $1$-form satisfying
\[
    d \omega_i = \theta \wedge \omega_i \, .
\]
Since $\omega_1$ is closed, it follows that $\theta \wedge df =0 $, and hence $\theta = g df$ for some $g \in \C(X)$. Therefore, the curvature of $\mathcal W$ is equal to $\Theta = d \theta = dg \wedge df$. If non-zero, it would define $\G$ implying that $f$ and $g$ are algebraically independent rational first integrals for $\G$. But this is impossible since we are assuming that $\trdeg(\C(X/G)) = 1$. We deduce that $\Theta$ is equal to zero, and any of the three foliations $\F_1, \F_2$, and $\F_3$, is transversely affine. Item (\ref{I:Cerveau trdeg 1}) follows.

If $\trdeg(\C(X/\G))=0$, then when $\eta$ is log-exact, as before, Lemma \ref{L:escritura} implies that every element in $\E_{\G}^1(\eta)\simeq \E_{\G}^1(0)$ is integrable. When $\eta$ is not log-exact the same holds true thanks to Proposition \ref{P:integravel}. Hence $\E_{\G}^1(\eta)$ is a two-dimensional $\C$-vector space, and the set of codimension one foliations containing $\G$ is described by $\mathbb P(\E_{\G}^1(\eta))(\C)$. Finally, Proposition \ref{P:Cerveau} implies that every foliation containing $\G$ is transversely affine. \qed

\subsection{Foliations generated by non-log exact eigenspaces}

\begin{cor}\label{C:Cerveau}
    Let $q$ be an integer strictly greater than $2$, and let  $\G$ be a codimension $q$ foliation on a projective manifold $X$. If there exists $\eta \in H^0(X,\Omega^1_{\G} \otimes \Mer_X)$ such that $\eta$ is not log-exact and $\E_{\G}^1(\eta)$ generates $\G$, then the general leaf of $\G$ is Zariski dense, every codimension one  foliation containing $\G$ is transversely affine, and the set of foliations containing $\G$ is parametrized by $\mathbb P^{q-1}(\C)$.
\end{cor}
\begin{proof}
    Since the codimension of $\G$ is at least $3$ by assumption, Corollary \ref{C:zariski dense} implies that $\C(X/\G) = \C$, that is, the general leaf of $\G$ is Zariski dense. Since $\eta$ is not log-exact, Proposition \ref{P:integravel} implies that any $1$-form in $\E_{\G}^1(\eta)$ is integrable. It follows that the set of codimension one foliations containing $\G$ is parametrized by $\mathbb P(\E_{\G}^1(\eta))(\C)= \mathbb P^{q-1}(\C)$ as claimed. Take any two linearly independent $1$-forms $\omega_1, \omega_2$ in $\E_{\G}^1(\eta)$ and considering the codimension two foliation $\G'$ defined by them. Theorem \ref{T:Cerveau} implies that the foliations defined by $s \omega_1 + t \omega_2$, for arbitrary $(s,t) \in \C^2-\{0\}$, are transversely affine.
\end{proof}

We finish this section by showing examples that fit the assumptions of Corollary \ref{C:Cerveau}.

\begin{example}
    Let $X_0$ be a simple abelian variety of dimension $q+1$ and consider a foliation $\G_0$ defined by a global holomorphic vector field. Since $X_0$ is simple, the leaves of $\G_0$ are Zariski dense. Note that for $\G_0$, $\E_{\G_0}^1(0) = H^0(X_0, N^*_{\G_0}) \subset H^0(X_0,\Omega^1_{X_0})$ has dimension $q$ and defines $\G_0$.

    Let $\varphi : X_0 \to X_0$ be the automorphism defined by multiplication by $-1$. Observe that $\varphi$ preserves $\G_0$ and all the codimension one foliations containing it. Set $\pi: X \dashrightarrow X_0$ equal to the composition of a  resolution of singularities of the quotient of $X_0$ by $\varphi$, and set $\G$ as the codimension $q$ foliation on $X$ induced by $\G_0$. Let $f \in \C(X_0)$ by any rational function on $X_0$ such that $\varphi^* f = -f$. If $\omega_1, \ldots, \omega_q \in H^0(X_0,N^*_{\G_0})= \E_{\G}^1(0)$ form a basis then $f \omega_1, \ldots, f\omega_q$ form a basis of
    \[
        \E_{\G_0}^1\left(\restr{\frac{df}{f}}{T_{\G_0}}\right) .
    \]
    The $\varphi$-invariance of $f \omega_i$ and of $\frac{df}{f}$ implies the existence of rational $1$-forms $\beta_1, \ldots, \beta_q$, $\eta$ on $X$ such that
    $f\omega_i = \pi^* \beta_i$ and $\frac{df}{f} = \pi^* \eta$. Moreover, $\beta_1, \ldots, \beta_q$ form a basis of $\E^1_{\G}(\restr{\eta}{T_{\G}})$. We claim that $\restr{\eta}{T_{\G}}$ is not log-exact. Indeed, if it were the case then it would exist a rational function $g \in \C(X)$ such that
    \[
        \restr{\eta}{T_{\G}} = \restr{\frac{dg}{g}}{T_{\G}}.
    \]
    The pull-back of $g$ under $\pi$ would be a $\varphi$-invariant rational function $\tilde g$ such that $\restr{\frac{df}{f} - \frac{d\tilde g}{\tilde g}}{T_{\G_0}}=0$. Therefore $f/\tilde g$ would be a rational first integral for $\G_0$ satisfying $\varphi^* f/\tilde g = - (f/ \tilde g)$, hence non-constant. This contradicts our initial choice of $\G_0$, proving that $\restr{\eta}{T_{\G}}$ is not log-exact and showing that the foliation $\G$ satisfies the assumptions of Corollary \ref{C:Cerveau}.
\end{example}

\section{Foliations without non-constant rational first integrals}\label{Section:trdeg0}

This section studies foliations of codimension at least two with field of rational first integrals equal to $\C$ and presents the proof of Theorem \ref{THM:Lie} from the introduction.

\subsection{Transversely Lie foliations}

\begin{prop}\label{P:subalgebra}
    Let $\G$ be a  transversely Lie foliation modeled over $\mathfrak g$ on a projective manifold $X$ with $\C(X/\G) = \C$. Let $\Omega$ be the $\mathfrak g$-valued rational $1$-form defining $\G$. If $\F$ is a foliation containing $\G$ then there exists a subalgebra $\mathfrak h \subset \mathfrak g$ of dimension $\dim \F - \dim \G$ such that $\F$ is defined by the $\mathfrak g/\mathfrak h$-valued rational $1$-form obtained by composing $\Omega$ with the natural projection $\mathfrak g \to \mathfrak g/\mathfrak h$.
\end{prop}
\begin{proof}
    Since $\G$ is transversely Lie we have that $\G$ is defined by $\E_\G^1(0)$. Since $\C(X/\G) = \C$, $\E_{\G}^1(0)$ is a dimension $q=\codim \G$ complex vector space. Notice that $\E_{\G}^k(0)$ generates $H^0(X, \wedge^k N^*_{\G} \otimes \Mer_X)$, see Lemma \ref{L:unique generation}.

    Let $\F$ be a codimension $k$ foliation containing $\G$ and let $\omega$ be a rational $k$-form defining $\F$. Lemma \ref{L:eta} implies the existence $\eta \in H^0(X, \Omega^1_{\G} \otimes \Mer_X)$ such that $\nabla(\omega) = \eta \otimes \omega$, where $\nabla$ is the partial connection on $\wedge^k N^*_{\G}$ induced by Bott's partial connection on $N^*_{\G}$. Lemma \ref{L:E independente} implies that $\eta$ is log-exact. Hence, after multiplication of $\omega$ by the inverse of a  primitive of $\eta$, we may assume that $\omega \in \E_{\G}^k(0) \simeq \wedge^k \mathfrak g^*$.

    The decomposability of $\omega$ as a rational $k$-form  implies the decomposability of $\omega$ as an element of $\wedge^k \mathfrak g^*$. Hence, the quotient $T_\F/T_{\G} \subset N_{\G}$ is generated by a vector subspace $\mathfrak h$ of $\mathfrak g$. The involutiveness of $T_{\F}$ implies that $\mathfrak h$ is a Lie subalgebra of $\mathfrak g$. The result follows.
\end{proof}

\subsection{Proof of Theorem \ref{THM:Lie}} Let $\G$ be a codimension $q$ foliation with $\C(X/\G)=\C$. If $q=2$ then the result follows from Theorem \ref{THM:pencil}. From now one, assume that $q >2$. Corollary \ref{C:q+1} implies the existence
of $\eta \in H^0(X, \Omega^1_{\G} \otimes \Mer_X)$ such that $\E^1_{\G}(\eta)$ generates $\G$.

Assume first that $\eta$ is log-exact. In this case the foliation $\G$ is transversely Lie and the result follows from Proposition \ref{P:subalgebra} above.

If instead $\eta$ is not log-exact then, as we are assuming $q > 2$,  Corollary \ref{C:Cerveau} implies that the set of codimension one foliations containing $\G$ is naturally isomorphic to $\mathbb P^{q-1}$ and that every codimension one foliation containing $\G$ is transversely affine. The claim about the higher codimension foliations follows from Lemma \ref{L:unique generation} as argued in the proof of Proposition \ref{P:subalgebra} taking into account the assumption $\C(X/\G) = \C$ .
\qed

\section{Foliations containing a transversely parallelizable foliation}\label{Section:transv_paral}

Before proceeding to the proofs of Theorem \ref{THM:main}, we establish the following intermediate result.

\begin{thm}\label{T:transversely projective}
    Let $\G$ be a codimension $q$ transversely parallelizable foliation on a projective manifold $X$ and let $\pi : X \dashrightarrow Y$ the rational map with irreducible general fiber to a smooth projective manifold $Y$ such that $\C(X/\G) = \pi^* \C(Y)$ given by Proposition \ref{P:corpo}. If $\F$ is a codimension one foliation containing $\G$ then either $\F$ is the pull-back under $\pi$ of a codimension one foliation on $Y$, or $\F$ is transversely projective.
\end{thm}
\begin{proof}
    Let $\pi : X \dashrightarrow Y$ be the rational map with irreducible fibers such that $\C(X/\G) = \pi^* \C(Y)$ given by Proposition \ref{P:corpo}. Let $F$ be (a projective smooth model of) a general fiber of $\pi$. The restriction of $\G$ to $F$ is a transversely Lie foliation with general leaf Zariski dense. 
    
    If an open subset of $F$ is contained in a leaf of $\F$ then, since $F$ is a general fiber,  $\F$ is the pull-back of a foliation on $Y$. 
    If instead $\F$ is generically transverse to $F$ then Proposition \ref{P:subalgebra} combined with Lemma \ref{L:LieTits} implies that $\restr{\F}{F}$ is a transversely projective foliation. Since $\restr{\F}{F}$ contains $\restr{\G}{F}$ which has Zariski dense general leaf, its general leaf is also Zariski dense. We are therefore in position to apply  \cite[Theorem 3.1]{MR4288634} to conclude that $\F$ is transversely projective.
\end{proof}

\subsection{Proof of Theorem \ref{THM:main}}
Let $S = \{ \F_1, \ldots, \F_{q+1}\} $ be a set of $q+1$ codimension one foliations, as in the statement of Theorem \ref{THM:main} and let $\G$ be the codimension $q$ foliation defined by their intersection. Corollary \ref{C:q+1} implies the existence of $\eta \in H^0(X, \Omega^1_{\G} \otimes \Mer_X)$ such that $\E^1_{\G}(\eta)$ defines $\G$. Recall that one of our assumptions is that $q\ge 3$.

If $\eta$ is not log-exact then Corollary \ref{C:Cerveau} implies that every foliation in $S$ is transversely affine and, in particular, transversely projective. 

If $\eta$ is log-exact then we can assume that $\eta =0$. Corollary \ref{C:quasitransvLie} implies that $\G$ is transversely parallelizable. Theorem \ref{T:transversely projective} implies that every foliation in $S$ is either transversely projective or, in the notation of the proof of Theorem \ref{T:transversely projective}, is the rational pull-back of a foliation on $Y$. Since $\dim Y = \trdeg \C(X/\G)$, at most $\trdeg \C(X/\G)$ foliations in $S$ are rational pull-backs from $Y$ as otherwise we would get a contradiction with the assumption that the intersection of any $q$ foliations in $S$ is equal to $\G$. Note that when $\trdeg \C(X/\G) = 1$, being a rational pull-back from $Y$ implies that the foliation is algebraically integrable and, in particular, transversely projective. The theorem follows. 
\qed

\section{Many implies infinetly many} \label{S:final}

This section is devoted to the proof of Theorem \ref{THM:infinitas}. For that, let $\G$ be a codimension $q$ foliation with $\C(X/\G)=\C$ and contained in a set $S$ of at least $q+1$ codimension one foliations. If $q = 2$ then the result follows from Theorem \ref{T:Cerveau}. From now on, assume that $q\ge 3$.

We will prove Theorem \ref{THM:infinitas} by assuming that the intersection of any $q$ foliations in $S$ has codimension $q$. This reduction will be possible because of the following lemma.

\begin{lemma}\label{lema:reducao}
Suppose that $\G$ is a foliation of codimension $q$ and $S$ a set of codimension one foliations containing $\G$. If $S$ has at least $\codim \G +1$ elements, then it has a subset $S'$ with the following property:
\begin{enumerate}
    \item \label{item:propriedade} The cardinality of $S'$, denoted by $\#S'$, is strictly greater than $1$, and the  intersection of any $\# S' - 1$ elements of $S'$ gives the same foliation, with codimension equal to $\# S'- 1$.
\end{enumerate}
\end{lemma}
\begin{proof}
    Consider $\mathcal{B}$ be the family of subsets $S$ having the following property:  $T\in \mathcal{B}$ if the foliation given by the intersection of the elements of $T$ has codimension strictly smaller than $\#T$. Note that $S$ itself is in $\mathcal{B}$. 

    Among all elements of $\mathcal{B}$, choose one with the smallest cardinality and denote it by $S'$. We prove that $S'$ also has property \ref{item:propriedade}. By definition, elements of $\mathcal{B}$ cannot be singletons. Moreover, if the codimension of the intersection is at most $\# S'-2$, then any subset of $S'$ with $\# S'-1$ elements will be in $\mathcal{B}$. This contradicts the minimality of the cardinality of $S'$.
\end{proof}

Consider the  field of rational first integrals field of $\G$. The proof of Theorem \ref{THM:infinitas} proceeds by analyzing three cases based on the transcendental degree of $\C(X/\G)$ over $\C$: zero, one and at least two.

\begin{prop}\label{P:infinitas_trdeg0}
    Theorem \ref{THM:infinitas} is valid when $\trdeg \C(X/\G) = 0$.
\end{prop}
\begin{proof}
Corollary \ref{C:q+1} implies the existence of $\eta \in H^0(X, \Omega^1_{\G} \otimes \Mer_X)$ such that $\E^1_{\G}(\eta)$ defines $\G$. 
If $\eta$ is not log-exact then the corollary follows from Proposition \ref{P:integravel}. If instead $\eta$ is log-exact then $\G$ is transversely Lie modeled over a non-abelian $q$-dimensional Lie algebra $\mathfrak g$. According to Theorem \ref{THM:Lie}, the codimension one foliations containing $\G$ are in bijection with the codimension one Lie subalgebras of $\mathfrak g$. If $\mathfrak h \subset \mathfrak g$ is a codimension one subalgebra, Lemma \ref{L:LieTits} implies the existence of a ideal $\mathfrak I$, contained in $\mathfrak h$, such that the quotient $\mathfrak g/ \mathfrak I$ is isomorphic a Lie subalgebra of $\mathfrak{sl}_2$. If the quotient $\mathfrak g/ \mathfrak I$ is one-dimensional then the corresponding foliation is defined by a closed rational $1$-form.  Hence if two of the codimension one foliations containing $\G$ are like that, we get a full pencil of foliations containing $\G$. If the quotient $\mathfrak g/ \mathfrak I$ is two-dimensional then it is isomorphic the affine Lie algebra and, again, we have a full pencil of foliations containing $\mathcal G$. Finally, if the quotient 
$\mathfrak g/ \mathfrak I$ is isomorphic to $\mathfrak{sl}_2$, we get a {\it conic} of foliations containing $\G$. 
\end{proof}

\begin{prop}\label{P:tr=1}
    Theorem \ref{THM:infinitas} is valid when $\trdeg \C(X/\G) = 1$.
\end{prop}
\begin{proof}
    Consider a foliation $\F\in S$ without rational first integral.
    By Theorem \ref{THM:main}, $\F$ admits a transversely projective structure, that is, there exists a triple of rational $1$-forms $(\omega_0,\omega_1,\omega_2)$ such that $\omega_0$ defines $\F$ and the rational 1-form on $X\times \C$
    \[
    \Omega = dz + \omega_0+z\omega_1+z^2\omega_2
    \]
    is integrable.

    We recall the notion of equivalent projective structures from \cite[Section 2.2]{MR4288634}: The two transversely projective structures $(\omega_0,\omega_1,\omega_2)$ and $(\tilde\omega_0,\tilde\omega_1,\tilde\omega_2)$ are equivalent if there are rational functions $f,g\in \C(X)$ such that
    \begin{align*}
        \tilde\omega_0 &= f\omega_0\\
        \tilde\omega_1 &= \omega_1-\frac{df}{f} + g\omega_0\\
        \tilde\omega_2 &= \frac{\omega_2}{f} +g\omega_1 +g_2\omega_0 -dg.
    \end{align*}
    Also, according to \cite[Lemma 2.2]{MR4288634}, a foliation admitting two non-equivalent transversely projective structures must admit a transversely affine structure.

    For the rest of this proof we will work under the triple $(\omega_0,\omega_1,\omega_2)$ for which $\omega_0\in\E^1_\G(0)$. Since $d\omega_0=\omega_0\wedge\omega_1$, we have
    \[
    0 = \nabla(\omega_0) = \restr{-\omega_1}{T_\G}\otimes\omega_0,
    \]
    which implies that $\restr{-\omega_1}{T_\G}=0$, that is, the distribution defined by $\omega_1$ also contains the foliation $\G$.

    Assume that $\F$ is transversely affine. We analyze two distinct cases:
    For the first, given by $d\omega_0\neq 0$, consider the pencil $\omega_0+t\omega_1$.
    For the second, given by $d\omega_0=0$, consider the pencil $\omega_0+t\;df$. In both scenarios, the resulting family of foliations, parametrized by $t\in\C$, contains $\G$.

    We now work under the hypothesis that $\mathcal{F}$ is not transversely affine.
    
    Notice that by fixing any $t\in \C$, the 1-form $\Omega_t=\omega_0+t\omega_1+t^2\omega_2$ is integrable, since $\Omega_t$ is the pull-back of $\Omega$ by the inclusion map $X = X\times \{t\} \hookrightarrow X\times \C$.
    Thus if, for each $t$, the foliation given by $\Omega_t$ contains the foliation $\G$, then the proposition is proved.
    Note also that the intersection of all foliations given by the $1$-forms $\Omega_t$, $t \in \C$,  is foliation defined by the $1$-forms $\omega_0,\omega_1,\omega_2$, which has codimension at most three.

    To conclude the proof of Proposition \ref{P:tr=1} it suffices to stablish the following statement: if for every triple $(\omega'_0,\omega'_1,\omega'_2)$ equivalent to $(\omega_0,\omega_1,\omega_2)$ we have $\restr{\omega'_i}{T_\G} \neq 0$, for some $i=0,1,2$, then $\F$ admits a transversely affine structure.

    Let $f$ be a non-constant element of $\C(X/\G)$. For a general fiber $F$ of $f$ consider the pullback of the foliations in the set $S$ by the inclusion $i_F:F\to X$. It follows that $\G_F:=i_F^*\G$ is a foliation without non-constant first integrals. Consider the pullback $i_F^*\F$ of the foliation $\F$. The foliation $\G_F$ is generated by $\E^1_{\G_F}(0)$ because $\G$ is generated by $\E^1_\G(0)$.
    Then, by the proof of Theorem \ref{THM:main}, $i_F^*\F$ admits a transversely projective structure $(\alpha_0,\alpha_1,\alpha_2)$ with $\alpha_i \in \E^1_{\G_F}(0)$. Remark that the distributions defined by $\alpha_0,\alpha_1,\alpha_2$ all contain $\iota^*_F\G$. 

    Consider now the pullback of the triple $(\omega_0,\omega_1,\omega_2)$ by $i_F$. Suppose there is no $(\omega_0', \omega_1',\omega_2')$ projective triple for $\F$, equivalent to $(\omega_0,\omega_1,\omega_2)$, in which each form $\omega_i'$  defines a distribution containing $\G$.
    Then $(i^*_F\omega_0,i^*_F\omega_1,i^*_F\omega_2)$ is not equivalent to $(\alpha_0,\alpha_1,\alpha_2)$, and thus we apply \cite[Lemma 2.2]{MR4288634} to obtain a transversely affine structure for $i^*_F\F$.

    By applying \cite[Theorem 3.1]{MR4288634}, which states that $\F$ admits a transversely affine structure if its restriction to a general fiber of a rational map does, we obtain a transversely affine structure for $\F$, contradicting our assumption.
\end{proof}

\begin{prop}\label{P:tr>1}
    Theorem \ref{THM:infinitas} is valid when $\trdeg \C(X/\G) \geq 2$.
\end{prop}
\begin{proof}
    If $\trdeg \C(X/\G)$ is at least two, then for any two algebraic independent elements $f_1,f_2$ of $\trdeg \C(X/\G)$ the pencil $\{f_1+tf_2|\; t\in\C\}$ is an infinite collection of first integrals for $\G$, each of them defines a different foliation.
\end{proof}

\begin{proof}[Proof of Theorem \ref{THM:infinitas}]
    It suffices to combine Propositions \ref{P:infinitas_trdeg0}, \ref{P:tr=1}, and \ref{P:tr>1}.
\end{proof}

\providecommand{\bysame}{\leavevmode\hbox to3em{\hrulefill}\thinspace}
\providecommand{\MR}{\relax\ifhmode\unskip\space\fi MR }
\providecommand{\MRhref}[2]{%
  \href{http://www.ams.org/mathscinet-getitem?mr=#1}{#2}
}
\providecommand{\href}[2]{#2}

\end{document}